\documentclass[reqno]{amsart}

\usepackage{amsmath,amssymb,amsthm}

\usepackage{tikz}
\usepackage{caption}
\usepackage{graphicx}
\usepackage{graphics}

\newtheorem{theorem}{Theorem}

\newtheorem{lemma}{Lemma}

\newcommand{\paren}[1]{\left( #1 \right)}
\newcommand{\abs}[1]{\left| #1 \right|}
\newcommand{\norm}[1]{\left\lVert #1 \right\rVert}

\begin{document}

\title[]{On differences of two harmonic numbers}

\author[]{Jeck Lim}
\address{Department of Mathematics, Caltech, Pasadena, CA 91125, USA}
\email{jlim@caltech.edu}

\author[]{Stefan Steinerberger}
\address{Department of Mathematics, University of Washington, Seattle, WA 98195, USA}
\email{steinerb@uw.edu}

\begin{abstract} We prove that the existence of infinitely many $(m_k, n_k) \in \mathbb{N}^2$ such that the difference of harmonic numbers $H_{m_k} - H_{n_k}$ approximates 1 well
$$ \lim_{k \rightarrow \infty} \left| \sum_{\ell = n}^{m_k} \frac{1}{\ell} - 1 \right|\cdot n_k^2 = 0.$$
This answers a question of Erd\H{o}s and Graham. The construction uses asymptotics for harmonic numbers, the precise nature of the continued fraction expansion of $e$ and a suitable rescaling of a subsequence of convergents. We also prove a quantitative rate by appealing to techniques of Heilbronn, Danicic, Harman, Hooley and others regarding $\min_{1 \leq n \leq N} \min_{m \in \mathbb{N}}\| n^2 \theta - m\|$.
\end{abstract}

\maketitle

\section{Differences of Harmonic Numbers}
\subsection{Introduction.} We are motivated by a problem of Erd\H{o}s and Graham \cite{erd}.
 \begin{quote}
     \textit{Choose $t = t(n)$ to be the least integer such that
     $$ \varepsilon_n = \sum_{k=n}^{t} \frac{1}{k} - 1 \geq 0.$$
     How small can $\varepsilon_n$ be? As far as we know this has not been looked at. It should be true that
     $ \liminf_{n} n^2 \varepsilon_n = 0$
     but perhaps $n^{2+\delta}\varepsilon_n \rightarrow \infty$ for every $\delta > 0$.
     } (P. Erd\H{o}s and R. Graham \cite[p.~41]{erd})
 \end{quote}

The problem is listed as Problem $\# 314$ in the list of Erd\H{o}s problems curated by Bloom \cite{bloom}. There is a natural heuristic: one has $t = (e+o(1))n$ and by the time the expression exceeds 1, it does so by an amount that is $\leq (1/e + o(1))/n$. However, the actual amount by which we exceed the threshold should be random and so we would expect it to be usually smaller. Assuming $\varepsilon_n$ to be a uniformly distributed random variable in the interval $[0, 1/(en)]$ leads to the prediction above. 
 
 \subsection{An elementary solution.}
 We address the first half of the problem. 
\begin{theorem}[Elementary version] For any $c>0$, there exists infinitely many $(m, n) \in \mathbb{N}^2$ with
$$1 \leq \sum_{\ell = n}^{m} \frac{1}{\ell} \leq 1 +  \frac{c}{n^2}.$$
\end{theorem}

The construction is fairly concrete: we describe how one can use a suitable subset of convergent fractions coming from the continued fraction expansion of $e$ to turn it into a pairs $(m, n) \in \mathbb{N}^2$ with the desired properties. The construction is  elementary, it does not require much beyond basics of the continued fraction expansion and is sufficient to resolve the question of Erd\H{o}s-Graham.

 \subsection{A refined solution.}
There is an ingredient in the construction that is slightly subtle: it deals with the problem of whether a particular real number can be well approximated by rational numbers of the form $a/b^2$. The trivial bound leads to Theorem 1. This question was studied by Danicic \cite{dan}, Harman \cite{harman}, Heilbronn \cite{heilbronn}, Hooley \cite{hooley} and others. Leveraging these techniques, we can say slightly more.

\begin{theorem}[Refined version] For every $\varepsilon > 0$, there exists infinitely many $(m, n) \in \mathbb{N}^2$ with
$$ \left| \sum_{\ell = n}^{m} \frac{1}{\ell} - 1 \right| \leq \frac{1}{n^2\left( \log n\right)^{(5/4) - \varepsilon}} .$$
\end{theorem}
One could further enforce that $\sum_{\ell=n}^m \frac{1}{\ell}>1$, but we chose not to for simplicity. This proof is non-constructive and relies on the existence of certain rational approximations of suitable form. One interesting aspect of this result is the comparison with the purely random case.  If $X_n$ is uniformly distributed in $[0,1/n]$, then
$$ \sum_{n=1}^{\infty} \mathbb{P}\left( X_n \leq \frac{1}{n^2(\log{n})^{1+\delta}} \right) = \sum_{n=1}^{\infty} \frac{1}{n(\log{n})^{1+\delta}} < \infty.$$
The random heuristic would predict the existence of a finite number of solutions and starts being inaccurate at that scale of resolution.\\

Regarding the second part of the question, as to whether $n^{2+\delta} \varepsilon_n \rightarrow \infty$, our arguments may be helpful in suggesting why one might expect this to be the case: the first part of our argument shows that any $(m, n) \in \mathbb{N}^2$ indexing a case where $H_{m} - H_{n}$ gets very close to 1 has to be connected to convergents $p/q$ coming from the continued fraction expansion of $e$ in a very explicit way: $(2m+1)/(2n-1) = p_k/q_k$ for some $k \in \mathbb{N}$. The numerator and denominator of convergents grows exponentially and one would expect the approximation quality to be polynomial in $k$ which predicts logarithmic approximation rates of the type shown above. Such logarithmic rates would then suggest that $n^{2+\delta} \varepsilon_n \rightarrow \infty$. However, ruling out an exceptional, very sparse subsequence with atypically good approximation properties would require entirely new arguments.

\subsection{Related results} We observe that summing over parts of the harmonic series has a long history (though these results all appear to be of a somewhat different flavor). A classic 1915 result of Theisinger \cite{theis} is that the sum $\sum_{k=1}^{n} 1/k$ can never be an integer. This result has since ben extended to many other settings, we refer to Nagell \cite{nagell}, Erd\H{o}s-Niven \cite{erdniv0, erdniv} and, more recently, Chen-Tang \cite{one}, Schinzel \cite{three} and Wang-Hong \cite{two} among others.

\section{Proofs}
\subsection{Outline} The argument comes in three parts. The first two parts are fairly general and work for an arbitrary $x>0$ (instead of only $x=1$). The third part requires additional knowledge and that is where we are going to set $x=1$. For now, let us fix some $x > 0$ and let us try to understand what would be required for
$$ \left| \sum_{k=n}^{m} \frac{1}{k}  - x \right| \leq \frac{\varepsilon}{n^2}$$
to have an infinite number of solutions. Since we are only interested in asymptotic results at scale $\sim n^{-2}$, the asymptotic expansion (see, for example, \cite{jam})
    $$ \sum_{\ell=1}^{n} \frac{1}{\ell}= \log{n} + \gamma + \frac{1}{2n} - \frac{1}{12n^2} + \mathcal{O}(n^{-4})$$
contains all the relevant information. Abbreviating
$$ f(n) = \log{n} + \gamma + \frac{1}{2n} - \frac{1}{12n^2},$$
the three steps of the proof are as follows.

\begin{enumerate}
    \item \textit{Asymptotics.}  Given an integer $n \in \mathbb{N}$, what can we say about the real numbers $m \in \mathbb{R}$ for which
    $$ \left| f(m) - f(n-1) - x \right| \leq \frac{\varepsilon}{n^2} \quad?$$
    Since there is a precise asymptotic expansion these real numbers $m$ can be described accurately in terms of $n$.
    \item \textit{Rational Approximation.} In order for the expression $f(m) - f(n-1)$ to actually correspond to the difference of two harmonic numbers, we need $m$ to be an integer. The question is thus whether any of the real numbers from the first part can ever end up being close enough to an integer to only cause a small additional error. Solutions of this problems happen to be related to convergents of the continued fraction expansion of $e^x$.
    \item \textit{Continued Fractions.} We use properties of continued fractions to get explicit constructions of such solutions: a result of Legendre implies that solutions of the relevant diophantine equation have to come from the continued fraction expansion: this suggests a natural one-parameter family of candidate solutions that, we show, contains suitable solutions.
\end{enumerate}

\subsection{Part 1. Asymptotics.}
 Let $x > 0$ be an arbitrary positive real. We think of $x$ as fixed and allow for all subsequent constants to depend on $x$. A trivial estimate
 $$ \sum_{k=n}^m \frac{1}{k} \sim \int_n^m \frac{dx}{x} = \log\left( \frac{m}{n} \right)$$
 shows that we will need $m = (1+o(1)) \cdot e^x \cdot n$ and this is correct up to leading order. Moreover, since, for $s = o(n)$,
\begin{align*}
 f(e^x n +s) - f(n-1) &= \log( e^x n + s) - \log{(n-1)} + \mathcal{O}(n^{-1})\\
 &= x + \frac{s}{e^x n} + \mathcal{O}(n^{-1}),
\end{align*}
where the implicit constant in $ \mathcal{O}(n^{-1})$ is independent of $n$. In order for the error to be small, one requires $m = e^x n + s$ with $|s| \leq c = \mathcal{O}(1)$. We shall now consecutively refine this and start with the first correction term at scale $n^{-1}$. The asymptotic expansion implies that, expanding now beyond the logarithmic and constant term to also include the term of order $n^{-1}$,
\begin{align*}
 f(e^x n +s) - f(n-1) &= \log\left(e^{x}n + s \right) - \log{(n-1)} \\
 &+ \frac{1}{2(e^{x}n + s)} - \frac{1}{2n-2} + \mathcal{O}(n^{-2}) \\
    &= x + \frac{e^{-x}(1 + e^{x} + 2s)}{2n} + \mathcal{O}(n^{-2}).
\end{align*}
We note that the error $\mathcal{O}(n^{-2})$ is uniform in the range $|s| \leq c$ which is the only relevant range since otherwise the leading order term $\log$ is already mismatched. In order for the distance to $x$ to be of order $\sim n^{-2}$, we require the first-order term to be at scale $\sim n^{-2}$. Moreover, in order for that distance to be $\sim \varepsilon/n^2$, we need the first-order term to be such that it matches and mostly cancels the other remaining $ \mathcal{O}(n^{-2})$ error term. Since the implicit constant in the $\mathcal{O}(n^{-2})$ term is uniformly bounded, we deduce that in order for the error to be $\sim n^{-2}$, we require
$$ s = - \frac{1+e^{x}}{2} + \frac{y}{n} \qquad \mbox{for some}~y \in \mathbb{R}$$
where $y$ is constrained to lie in a compact interval (whose size could be bounded in terms of the implicit constant in the remaining error term). We now return to the expression $f(e^x n +s) - f(n-1)$ with $s$ being of this form, expand it up to second order. After some computation,
$$ f\left(e^x n - \frac{1+e^{x}}{2} + \frac{y}{n}\right) - f(n-1) = x + \frac{24 e^{-x} y + e^{-2x} - 1}{24} \frac{1}{n^2} + \mathcal{O}(n^{-3}).$$
This shows that the only way the approximation can be as good as $o(n^2)$ is if $y$ is within distance $o(1)$ of
$$ y^{*} = \frac{e^{x} - e^{-x}}{24} = \frac{\sinh{x}}{12}.$$
The expression also shows that, as $n \rightarrow \infty$, 
$$ \left|f\left(e^x n - \frac{1+e^{x}}{2} + \frac{y}{n}\right) - f(n-1) - x\right| \leq \frac{\varepsilon}{n^2}$$
holds if $|y - y^*| \leq e^{x} \varepsilon/2$. 
If there is an infinite sequence of integers $n_k \in \mathbb{N}$ with
$$e^x n_k - \frac{1+e^{x}}{2} + \frac{y_{k}}{n_k} \in \mathbb{N} \qquad \mbox{and} \qquad \lim_{k \rightarrow \infty} y_k = y^*,$$
then there is an infinite number of approximations for which the distance between the harmonic number and $x$ is $o(n^{-2})$.

\subsection{Part 2. Rational Approximation.} The remaining question is whether there are infinitely many integers $n \in \mathbb{N}$ such that
\begin{align*}
e^x n - \frac{1+e^x}{2} + \frac{y}{n} = m \in \mathbb{N}
 \end{align*}
 with $y \in \mathbb{R}$ close to $y^* = \sinh{x}/12$.  

 \begin{lemma} If $n \in \mathbb{N}$ and
 $$ e^x n - \frac{1+e^x}{2} + \frac{y}{n} =m \in \mathbb{N},$$
     then  either $|y| \geq 1/8$ or $(2m+1)/(2n-1) \in \mathbb{Q}$ is a convergent coming from the continued fraction expansion of $e^x$.
 \end{lemma}

This result will be used as follows: one is only interested in the cases where $y$ is close to $y^*$. The relevant question is thus when
$$ y^* =  \frac{\sinh{x}}{12} < \frac{1}{8} \quad \mbox{which requires} \quad x < \log\left(\frac{3 + \sqrt{13}}{2}\right) \sim 1.1947\dots$$
As long as this restriction on $x$ is satisfied, any such solution  
\begin{enumerate}
    \item either has $|y| \geq 1/8$ and $|y - y^*|$ is uniformly bounded from below  
    \item  or $(2m+1)/(2n-1)$ has to be a of a very peculiar form.
\end{enumerate}
In the first case, no approximation better than $\geq c \cdot n^{-2}$ is possible.

 \begin{proof}[Proof of the Lemma]   
The equation
 $$ m =  e^x n - \frac{1+e^x}{2} + \frac{y}{n}$$
can be rewritten as
$$ e^x = \frac{2m+1}{2n-1} - \frac{2y}{n (2n-1)}$$
and thus
$$ \left| e^x - \frac{2m+1}{2n-1} \right| = \frac{4|y|}{2n (2n-1)} < \frac{1}{2}\frac{8|y|}{(2n-1)^2}.$$
We invoke a classical result of Legendre (see, for example, \cite[Theorem 1.8]{bugeaud}) stating that if $\alpha >0$ is real and $p,q$ are positive integers such that
$$ \left| \alpha - \frac{p}{q} \right| \leq \frac{1}{2} \frac{1}{q^2},$$
then $p/q$ is a convergent derived from the continued fraction of $\alpha$. Hence, if $|y| \leq 1/8$, any such solution must come from the continued fraction expansion of $e^x$. 
 \end{proof}

\subsection{Convergents for $e$} We now focus on the case of $x=1$ and the case of convergents that approximate $e$. Some of the arguments may generalize to cases of $x$ such that one has enough information about the continued fraction expansion of $e^x$. Some examples are
\begin{align*}
    e^{1/2} &= [1; 1, 1, 1, 5, 1, 1, 9, 1, 1, 13, 1, 1, 17, 1, 1, 21, 1, 1, \dots] \\
 e^{1/3} &= [1; 2, 1, 1, 8, 1, 1, 14, 1, 1, 20, 1, 1, 26, 1, 1, 32, 1, 1, \dots] \\
 e^{1/4} &= [1; 3, 1, 1, 11, 1, 1, 19, 1, 1, 27, 1, 1, 35, 1, 1, 43, 1, 1, \dots]  
\end{align*}
 and, more generally, (see Osler \cite{osler}) the identity
$$ e^{1/k} = [ \overline{1,(k-1) + 2k m , 1}],$$
where the bar refers to periodic repetition and $m$ runs from 0 to $\infty$. A similar identity for $e^{2/(2n+1)}$ is given by Olds \cite{olds}.
We focus on
$$ e = [2; 1, 2, 1, 1, 4, 1, 1, 6, 1, 1, 8, 1, 1, \dots]$$
 This gives rise to a sequence $(p_k/q_k)_{k=1}^{\infty}$ of convergent fractions
 $$ \frac{2}{1}, \frac{3}{1}, \frac{8}{3}, \frac{11}{4}, \frac{19}{7}, \frac{87}{32}, \frac{106}{39}, \frac{193}{71}, \dots$$
We will use the subsequence $p_{3k+2}/q_{3k+2}$
$$ \frac{3}{1}, \frac{19}{7}, \frac{193}{71}, \frac{2721}{1001}, \dots$$
which has a number of useful properties.

\begin{lemma}
    For the subsequence of convergents  $p_{3k+2}/q_{3k+2}$, we have that
\begin{enumerate}
    \item $p_{3k+2}$ and $q_{3k+2}$ are odd
    \item and, for some $r_{3k+2} > 0$ 
    $$ e - \frac{p_{3k+2}}{q_{3k+2}} = (-1)^{k+1} \cdot\frac{ r_{3k+2}}{q_{3k+2}^2}$$
    \item and, moreover, $\frac{1}{2k+4}\leq r_{3k+2} \leq \frac{1}{2k+2}$.
\end{enumerate}
\end{lemma}
\begin{proof} The subsequence $p_{3k+2}/q_{3k+2}$ corresponds to taking three consecutive continued fraction expansion coefficients of the form $(2\ell, 1, 1)$. 
For the first claim, we prove more generally by induction that
\begin{enumerate}
    \item $p_{6k+i}$ is odd for $i\in \{2,4,5,6\}$ and is even for $i\in \{1,3\}$
    \item and, $q_{6k+i}$ is odd for $i\in \{1,2,3,5\}$ and is even for $i\in \{4,6\}$.
\end{enumerate}
This is easily seen from the recurrence
$$ p_{n} = a_n p_{n-1} + p_{n-2} \qquad \mbox{as well as} \qquad  q_{n} = a_n q_{n-1} + q_{n-2},$$
and the fact that $a_{3k+3}$ is even while $a_{3k+1},a_{3k+2}$ are odd.
The second part of the Lemma is a more general fact for convergents. More precisely, one has
$$ \frac{p_k}{q_k} - \frac{p_{k-1}}{q_{k-1}} = \frac{(-1)^{k}}{q_k q_{k-1}}$$
and thus, by telescoping
$$ e = 2 + \sum_{k=0}^{\infty} \frac{(-1)^{k+1}}{q_k q_{k+1}}.$$
Hence, convergents alternatingly over- and underestimate the number they are approximating.  The subsequence $p_{3k+2}/q_{3k+2}$ corresponds to step sizes of length 3 and since 3 is odd, this alternating property remains preserved. As for the third property, we recall the general inequality (see, for example, Bugeaud \cite{bugeaud})
$$ \frac{1}{q_k (q_{k+1} + q_k)} < \left|e - \frac{p_k}{q_k} \right| < \frac{1}{q_k q_{k+1}}$$
from which we derive
$$ \left|e - \frac{p_k}{q_k} \right| q_k^2 < \frac{q_k}{q_{k+1}} = \frac{q_k}{a_{k+1} q_k + q_{k-1}} \leq \frac{1}{a_{k+1}}.$$
In our case, we have $a_{3k+3} = 2k+2$ and the desired inequality follows. The lower bound also follows similarly.
\end{proof}

\subsection{Proof of the Theorem}
Using that $p_{3k+2}, q_{3k+2}$ are odd, we can now, for every $d \in 2 \mathbb{N} + 1$, find a $m,n \in \mathbb{N}$ satisfying
$$ 2m+1 = d \cdot p_{3k+2} \qquad \mbox{and} \qquad 2n-1 = d \cdot q_{3k+2}.$$
Then, plugging in, we get that
\begin{align*}
     e &=  \frac{p_{3k+2}}{q_{3k+2}} +  (-1)^{k+1} \cdot\frac{ r_{3k+2}}{q_{3k+2}^2} \\
     &= \frac{2m+1}{2n-1} + (-1)^{k+1} \cdot \frac{r_{3k+2} \cdot d^2}{(2n-1)^2} \\
     &= \frac{2m+1}{2n-1} + (-1)^{k+1} \cdot \frac{r_{3k+2} \cdot d^2}{n(2n-1)} \frac{n}{2n-1}.
\end{align*}

The question is thus whether we can find a suitable choice $k,d$  such that
$$ (-1)^{k+1} \cdot r_{3k+2} \cdot d^2 \cdot \frac{n}{2n-1} \qquad \mbox{can get close to}\qquad -2y^* = -\frac{\sinh{(1)}}{6}.$$
The sign forces $k$ to be even, the factor $n/(2n-1)$ converges to $1/2$ contributes only lower order effects.

\begin{lemma}
    Let $k \in 2\mathbb{N}$ be sufficiently large. There exists $d \in 2 \mathbb{N} + 1$ such that
    $$ \frac{1}{100\sqrt{k}}\leq r_{3k+2} \cdot d^2 \cdot \frac{n}{2n-1} - \frac{\sinh{(1)}}{6} \leq \frac{100}{\sqrt{k}}.$$
\end{lemma}
\begin{proof}
    If we were allowed to choose real numbers, we could choose
    $$ d^* = \left(\frac{2n-1}{n} \frac{1}{r_{3k+2}} \frac{\sinh{(1)}}{6} \right)^{1/2} \approx \left( \frac{0.097}{r_{3k+2}} \right)^{1/2}$$
    to force the expression to be 0. This expression grows roughly at the rate of $d^* \sim \sqrt{k}$. However, we are forced to choose $d$ to be an odd integer: picking $d$ to be the closest odd integer to $d^*+2$, we have $d^*+1\leq d\leq d^*+3$ so that
   $$ \frac12 d^* r_{3k+2} \leq r_{3k+2} \cdot d^2 \cdot \frac{n}{2n-1} - \frac{\sinh{(1)}}{6} \leq 7 d^* r_{3k+2}.$$
   Using the estimate $\frac{1}{2k+4}\leq r_{3k+2}\leq \frac{1}{2k+2}$, the result follows.
 \end{proof}

\subsection{Proof of Theorem 1}
\begin{proof} The previous Lemma shows that we can find suitable parameters for the relevant quantity to
 get arbitrarily close to $-2 y^*$. The argument comes with the quantitative rate $1/(200\sqrt{k})\leq y-y^*\leq 50/\sqrt{k}$
 and shows that solutions that are so constructed, satisfy $m_k \sim \sqrt{k} \cdot p_{3k+2}$ and $n_k \sim \sqrt{k} \cdot q_{3k+2}$ and satisfy
 $$\frac{1}{1000n_k^2\sqrt{k}}\leq f(m_k)-f(n_k-1)-1\leq \frac{1000}{n_k^2\sqrt{k}}.$$
 The difference between $f(m_k)-f(n_k-1)$ and $H_{m_k}-H_{n_k-1}$ is $\mathcal{O}(1/n_k^4)$, negligible compared to $1/(n_k^2\sqrt{k})$, so we have
  $$ 1 \leq  \sum_{\ell = n_k}^{m_k} \frac{1}{\ell}  \leq 1+\frac{1001}{\sqrt{k}} \frac{1}{n_k^2}.$$
\end{proof}

\section{Proof of Theorem 2}
The previous section derived approximation bounds from each continued fraction expansion convergent $p_{3k+2}/q_{3k+2}$: for each such convergent, we identified a suitable $d \in 2\mathbb{N}+1$ with a uniform bound. One might be inclined to believe that some values of $k$ should give rise to better values of $d$: instead of fixing $k$ and finding $d$, we will interpret it as a diophantine problem jointly in $(d,k)$.

\subsection{Sharper estimates.}
We first give a slightly sharper estimate for $r_{3k+2}$.
\begin{lemma} We have
    $$r_{3k+2}^{-1} = 2k + 3 + \mathcal{O}(1/k).$$
\end{lemma}
\begin{proof}
From the continued fraction expansion, we have
$$e=\frac{p_{3k+1}+w_k p_{3k+2}}{q_{3k+1}+w_k q_{3k+2}},$$
where $w_k=[2k+2;1,1,2k+4,1,1,2k+6,\ldots]$. Then
\begin{align*}
    \left| e-\frac{p_{3k+2}}{q_{3k+2}}\right| &= \left| \frac{p_{3k+1}+w_k p_{3k+2}}{q_{3k+1}+w_k q_{3k+2}}-\frac{p_{3k+2}}{q_{3k+2}}\right|\\
    &= \frac{1}{q_{3k+2}(q_{3k+1}+w_k q_{3k+2})}= \frac{r_{3k+2}}{q_{3k+2}^2}.
\end{align*}
Let $c_k=q_{3k+1}/q_{3k+2}$, then $0<c_k<1$ and $r_{3k+2}^{-1}=c_k+w_k$.
From the recurrence relations, we have
\begin{align*}
    q_{3k+3} &=(2k+2)q_{3k+2}+q_{3k+1}=(2k+2+c_k)q_{3k+2} \\
    q_{3k+4} &=(2k+3+c_k)q_{3k+2} \\
    q_{3k+5} &=(4k+5+2c_k)q_{3k+2}
\end{align*}
Thus $$c_{k+1}=\frac12+\frac{1}{2(4k+5+2c_k)}=\frac12+\mathcal{O}(1/k),$$ 
and we also have $c_k=1/2+\mathcal{O}(1/k)$. Moreover,
$$w_k=2k+2+\frac{1}{1+\frac{1}{1+\frac{\mathcal{O}(1)}{k}}}=2k+\frac52+\mathcal{O}(1/k).$$
Therefore, $r_{3k+2}^{-1}=2k+3+\mathcal{O}(1/k)$.
\end{proof}

In particular, we have now a fairly precise control on the approximation rate
$$ \left| e-\frac{p_{3k+2}}{q_{3k+2}}\right| = \frac{1}{q_{3k+2}^2} \frac{1}{2k + 3 + \mathcal{O}(1/k)}.$$
The next step of the argument consists in finding good approximations of a rational number $p/q$ by rational numbers of the form $m/n^2$.

\subsection{Counting Solutions.}

For $x\in \mathbb{R}$, write $e(x)=e^{2\pi ix}$ and $\lVert x\rVert$ to be the distance from $x$ to the nearest integer. We will make use of a specific form of the Erd\H{o}s-Tur\'an theorem to turn the number of approximate solutions into an estimate involving exponential sums. The specific of Erd\H{o}s-Tur\'an that we use can be found in the book of Montgomery \cite{mont}.

\begin{lemma}[Erd\H{o}s-Tur\'an, \cite{mont}] \label{lem:etcount}
    Let $x_1,\ldots,x_N$ be real, let $I=[a,b] \subset \mathbb{T} \cong [0,1]$ be an interval of length $\delta=b-a<1$. For any $L \in \mathbb{N}$,
    \begin{align*}
            |\#\{1 \leq n \leq N: x_n \in I \pmod{1} \}-N\delta| &\leq \frac{N}{L+1} + E,
    \end{align*}
where
$$ E = 2 \sum_{m=1}^{L}\left(\frac{1}{L+1}+\min(b-a, \frac{1}{\pi m}) \right)\left| \sum_{n=1}^{N} e( m x_n) \right|.$$
\end{lemma}
We will use the more compact version that does not distinguish between $m$ small and $m$ large and bound the error uniformly by
$$ E \leq 2 \left( \frac{1}{L+1} + \delta\right) \sum_{m=1}^{L} \left| \sum_{n=1}^{N} e( m x_n) \right|.$$
We sometimes abbreviate the exponential sum as $S_m=\sum_{n=1}^{N} e( m x_n)$. Erd\H{o}s-Tur\'an then implies the following Lemma (see \cite[Chapter 3]{baker}).

\begin{lemma}\label{lem:count}
    Let $q\geq 1$ and $p$ be coprime to $q$. Let $b\in\mathbb{N}$, $a\in \mathbb{Z}$ and $r\in \mathbb{R}$. Then, for $\eta>0,\delta\in (0,1/2), N\geq 1$, the number of $n\in \mathbb{N}$ with
    \begin{align} \label{eqn:countlem}
    \left\lVert \frac{pn^2}{q}-r\right\rVert<\delta,\qquad n\leq N,\qquad n\equiv a\pmod{b}
    \end{align}
    is
    $$\frac{2N\delta}{b} (1+\mathcal{O}(N^{-\eta}))+\mathcal{O}\left(\frac{N^{1+2\eta}}{\delta^{\eta}}\left( \frac{\log q}{N}+\frac{1}{q}+\frac{q\delta\log q}{N^2}\right)^{\frac12}\right).$$
    Here the implied constants depend only on $b,\eta$.
\end{lemma}
\begin{proof} We use the interval $I = [-\delta, \delta]$ and consider the set of points
$$y_n=\frac{pn^2}{q}-r$$
for $n=1,\ldots,N$. We let $\eta > 0$ be arbitrary and set $L=N^\eta /\delta$. The set of points $x_i$ will be exactly all $y_n$ with $n\equiv a\pmod{b}$. The cardinality of the set will be approximately $N/b + \mathcal{O}(1)$. Then, with Erd\H{o}s-Tur\'an, we have that the number of solutions to (\ref{eqn:countlem}) satisfies
    \begin{align*}
    \left| \mbox{solutions} -\frac{2\delta N}{b} \right| &\leq \frac{N}{L}+2\left(\frac{1}{L+1}+\delta\right)\sum_{m=1}^L |S_m| \leq \delta \left( N^{1-\eta} + 3 \sum_{m=1}^L |S_m|\right).
    \end{align*}
    Recall that $S_m=\sum_{n=1}^N e(mx_n)$, then 
    \begin{align*}
        |S_m|^2 &= \sum_{n_1,n_2} e(m(x_{n_1}-x_{n_2}))= \sum_{n_1,n_2} e\paren{\frac{mp}{q}(n_1^2-n_2^2)}= \sum_{u,v} e\paren{\frac{mp}{q}uv},
    \end{align*}
    where $u=n_1+n_2,v=n_1-n_2$, and the sum is over an appropriate set of pairs $(u,v)$. Note that $u,v$ each assume less than $2N$ possible values. 
    Recalling the restriction $n\equiv a\pmod{b}$ we see that $n_1 = k_1 b + a$ and $n_2 = k_2b + a$ and the variables $k_1, k_2$ range over $0 \leq k_1, k_2 \leq N/b + \mathcal{O}(1)$.
    We will use the triangle inequality once more
    $$ \left| \sum_{u,v} e\paren{\frac{mp}{q}uv} \right| \leq \sum_v \left| \sum_{u} e\paren{\frac{mp}{q}uv} \right|,$$
    where the values attained by $u$ in the inner sum depend on the value of $v$ in the outer sum. Fixing one particular value of $v$,
    $$ v = n_1 - n_2 = b (k_1 - k_2),$$
    forces $k_1$ and $k_2$ to be a fixed difference, say $d = v/b$, apart. Then the admissible values of $u$ are
    $$ u = n_1 + n_2 = (k_1+k_2)b + 2a = 2 k_1 b + 2a + d$$
    and thus is an arithmetic progression of length $\leq n/b + \mathcal{O}(1)$ and common difference $2b$. Exponential sums over arithmetic progressions are really exponential sums of the form $\sum_{x=U+1}^{U+V} e(\gamma x)$ for some $\gamma\in\mathbb{R}$. The two easy estimates are
    $$ \left| \sum_{x=U+1}^{U+V} e(\gamma x) \right| \leq V \qquad \mbox{and} \qquad  \left| \sum_{x=U+1}^{U+V} e(\gamma x) \right| \leq \frac{2}{\|\gamma\|},$$
where the second estimate follows from explicity computing and bounding the geometric series.   Applying this, we have (treating $v = 0$ and $v \neq 0$ separately)
    \begin{align*}
        |S_m|^2 &\leq \sum_{v=-N}^N \abs{\sum_u e\paren{\frac{mpv}{q}u}} \leq N + 4\sum_{v=1}^N \min\paren{N,\norm{\frac{mpv}{q} 2b}^{-1}}\\
        &\leq 5\sum_{v=1}^N \min\paren{N,\norm{\frac{mpv}{q} 2b}^{-1}}.
    \end{align*}
    By Cauchy-Schwarz,
    \begin{align*}
        \sum_{m=1}^L |S_m| &\leq \sqrt{L\sum_{m=1}^L |S_m|^2} \leq \sqrt{5L\sum_{m=1}^L \sum_{v=1}^N \min\paren{N,\norm{\frac{2bmpv}{q}}^{-1}}}.
    \end{align*}  
 %   For any positive integer $z\leq LN\leq N^{3+\eta}$ {\color{red} (um I guess we need some bound on $L$ here)}, the number of divisors it has is $z^{o(1)}\leq N^{\eta}$ {\color{red} (citation needed?)}. So the above expression is at most
 %   $$\sqrt{LN^{\eta}\sum_{z=1}^{LN} \min\paren{N,\norm{\frac{2pz}{q}}^{-1}}}.$$
At this point, we note that the terms $1 \leq m \leq L$ and $1 \leq v \leq N$ appear only as a product $mv$. This suggests replacing the two sums by a sum $1 \leq z = mv \leq LN$. We note that each fixed integer $z = mv$ can arise in possibly more than one way (meaning through various combinations of different $m$ and $v$). An easy upper bound is the number of divisors $d(z)$ of $z$: using a bound, valid for every $\varepsilon > 0$ with a sufficiently large constant $c_{\varepsilon}$,
$$ \max_{1 \leq z \leq LN} d(z) \leq c_{\varepsilon} (NL)^{\varepsilon}.$$
Hence
$$ \sum_{m=1}^L \sum_{v=1}^N \min\paren{N,\norm{\frac{2bmpv}{q}}^{-1}} \leq  c_{\varepsilon} (NL)^{\varepsilon} \sum_{z=1}^{LN} \min\paren{N,\norm{\frac{2bp z}{q}}^{-1}}$$
 
    Now let $2bp/q=p'/q'$ be its simplest form, so that $q/(2b)\leq q'\leq q$. Then 
    $$\sum_{z=1}^{q'} \min\paren{N,\norm{\frac{p'z}{q'}}^{-1}}\leq N+\sum_{k=1}^{q'-1}\frac{q'}{k}\leq N+q'\log q'.$$
    Thus by summing in groups of $q'$, we have
    \begin{align*}
        \sum_{z=1}^{LN} \min\paren{N,\norm{\frac{2pz}{q}}^{-1}} &\leq (N+q\log q)\paren{\frac{NL}{q'}+1}\\
        &\leq 4b\left( \frac{N^2L}{q}+NL\log q+q\log q\right).
    \end{align*}
    Therefore, choosing $\varepsilon = \eta/100$,
    \begin{align*}
        \delta\sum_{m=1}^L |S_m| &\leq 20 b \delta  c_{\varepsilon} (NL)^{\varepsilon} \left(\frac{N^2L^2}{q}+NL^2\log q+qL\log q\right)^{\frac12}\\
        &\leq c_{\eta,b} \frac{N^{1+2\eta}}{\delta^{\eta}}\paren{\frac{1}{q}+\frac{\log q}{N}+\frac{q\delta\log q}{N^2}}^{\frac12}.
    \end{align*}
\end{proof}

\begin{lemma} 
    Let $\varepsilon>0$ and $\alpha>0$ be irrational. Let $a,a',b,b'$ be integers with $b,b'>0$. Then there are infinitely many pairs of integers $m,n>0$ such that $n\equiv a\pmod{b}$ and $m\equiv a'\pmod{b'}$ and
    $$\left| \alpha-\frac{m}{n^2}\right| < \frac{1}{n^{5/2-\varepsilon}}.$$
\end{lemma}
\begin{proof}
    Let $p/q$ be a convergent from the continued fraction expansion of $\alpha$. Then
    $$ \left| \alpha - \frac{p}{q} \right| \leq \frac{1}{q^2}.$$
    The remainder of the argument is concerned with approximating $p/q$ by a rational number of the form $m/n^2$ and for this we use Lemma~\ref{lem:count}. We fix $\varepsilon > 0$ as well as $\eta=\varepsilon/4$ and choose $N$ such that $q^2=N^{5/2-\varepsilon}$ and set $\delta=N^{-1/2+\varepsilon}$. 
    By Lemma~\ref{lem:count}, the number of $n\leq N$ with $n\equiv a\pmod{b}$ satisfying
    \begin{align}
        \norm{\frac{pn^2}{qb'}-\frac{a'}{b'}} < \frac{\delta}{b'},\label{eqn:count}
    \end{align}
    is $2N^{1/2+\varepsilon}/bb'+O(N^{1/2+\varepsilon/2})$ which implies the existence of at least one such solution once $q$ (and thus $N$) is sufficiently large. In particular, it has arbitrarily many solutions as $q\to\infty$. Let $m'$ be the integer closest to $pn^2/(qb')- a'/b'$ and $m=a'+b'm'$. Then 
    $$\abs{\frac{pn^2}{q}-m}<\delta.$$
    Since $|\alpha-p/q|<1/q^2$, we have
    \begin{align*}
        \left| \alpha-\frac{m}{n^2}\right| &\leq \left|\alpha-\frac{p}{q}\right| + \left|\frac{p}{q}-\frac{m}{n^2}\right|\leq \frac{1}{q^2}+\frac{\delta}{n^2}\leq\frac{2}{n^{5/2-\varepsilon}}.
    \end{align*}
\end{proof}

\subsection{Proof of Theorem 2}
The proof of the main result uses Lemma 7 applied to the number $\alpha = 3/\sinh{(1)}$. This requires us to argue that $\sinh{(1)}$, or equivalently, $2 \sinh{(1)} = e - e^{-1}$ is irrational. If $e - 1/e = p/q$, then $e$ would be a root of the polynomial $p(x) = q x^2 - px - q$ which contradicts $e$ being transcendental. 

\begin{lemma}
    Let $\varepsilon>0$. There exists infinitely many pairs $(k,d)$ with $k$ even and $d$ odd such that
    $$ \left| r_{3k+2} \cdot d^2 \cdot \frac{n}{2n-1} - \frac{\sinh{(1)}}{6} \right| < \frac{1}{k^{5/4-\varepsilon}}.$$
\end{lemma}
\begin{proof}
    Recall that $2n-1=d \cdot q_{3k+2}$ and $ q_{3k+2}$ grows slightly faster than exponentially in $k$. Moreover, $r_{3k+2} \sim 1/k$ and thus $d \sim \sqrt{k}$. Thus
    $$ \left|r_{3k+2} \cdot d^2 \cdot \frac{n}{2n-1} - r_{3k+2} \cdot d^2 \cdot\frac{1}{2}  \right| \leq \frac{100}{n} \ll \frac{1}{k^{10}}.$$
     We can thus simplify the problem by instead asking for an approximation where $n/(2n-1)$ has been replaced by $1/2$.
     Since $d^2\sim k$, we wish to find $(k,d)$ with
    $$\left| r_{3k+2} \cdot d^2 - \frac{\sinh{(1)}}{3} \right| \lesssim \frac{1}{d^{5/2-\varepsilon}}.$$
    By Lemma 7, we can find infinitely many pairs $(k',d)$ with $k'\equiv 3\pmod{4}$ and $d$ odd satisfying
    $$\left| \frac{3}{\sinh{(1)}}-\frac{k'}{d^2}\right| \lesssim \frac{1}{d^{5/2-\varepsilon}}.$$
    Let $k$ be the even number so that $2k+3=k'$. Then the above is equivalent to
    $$\left| \frac{d^2}{2k+3}-\frac{\sinh{(1)}}{3}\right| \lesssim \frac{1}{d^{5/2-\varepsilon}}.$$
    Finally, let $r_{3k+2}^{-1}=2k+3+\varepsilon_2$ with $\varepsilon_2=\mathcal{O}(1/k)$. Then
    $$\left|r_{3k+2}d^2-\frac{d^2}{2k+3}\right|=\mathcal{O}\left( \frac{d^2\varepsilon_2}{k^2}\right)=\mathcal{O}\left( \frac{1}{d^3}\right)$$
    and then, via triangle inequality,
    $$\left|r_{3k+2}d^2- \frac{\sinh{(1)}}{3} \right| \leq 
    \left|r_{3k+2}d^2- \frac{d^2}{2k+3} \right| + \left| \frac{d^2}{2k+3}-\frac{\sinh{(1)}}{3}\right| \lesssim \frac{1}{d^{5/2-\varepsilon}}.$$
\end{proof}

Since $\log q_k \sim k \log{k}$, this leads to the estimate
$$\left| \sum_{\ell = n_k}^{m_k} \frac{1}{\ell} - 1 \right|\cdot n_k^2 \leq \frac{c}{\log^{5/4-\varepsilon} n_k}.$$

\vspace{10pt}
\textbf{Acknowledgment.} 
JL was partially supported by an NUS Overseas Graduate Scholarship. SS was partially supported by the NSF (DMS-212322) and is grateful for discussions with Vjeko Kovac.


\begin{thebibliography}{10}

\bibitem{baker} R. C. Baker, Diophantine inequalities, London Math. Soc. Monographs N.S.I, Oxford Science Publications (1986).

\bibitem{bloom} T. Bloom,  \textsc{www.erdosproblems.com}, March 2024.

\bibitem{bugeaud} Y. Bugeaud, Approximation by algebraic numbers, Cambridge University Press, 2004.

\bibitem{one} Y. Chen and M. Tang, On the Elementary Symmetric Functions of 1, 1/2,…, 1/n. The American Mathematical Monthly, 119 (2012), p.862--867.

\bibitem{dan}  I. Danicic, An extension of a theorem of Heilbronn, Mathematika 5 (1958), 30-37.

\bibitem{erd} P. Erd\H{o}s and R. Graham, Old and new problems and results in combinatorial number theory. Monographies de L'Enseignement Mathematique, Universit\'e de Gen\`eve, 1980. 

\bibitem{erdniv0}  P. Erd\H{o}s and I. Niven, On certain variations of the harmonic series, Bull. Amer. Math. Soc. 51 (1945), 433--436

\bibitem{erdniv} P. Erd\H{o}s and I. Niven, Some properties of partial sums of the harmonic series, Bull. Amer. Math. Soc. 52 (1946), 248--251.

\bibitem{harman} G. Harman, A problem of Hooley in Diophantine approximation, Glasgow Mathematical Journal, vol. 38, no. 3 (1996), pp. 299–308.

\bibitem{heilbronn} H. A. Heilbronn, On the distribution of the sequence $\theta n^2$ (mod 1), Quart. J. Math. Oxford Ser. 2,19 (1948), 249-256.

\bibitem{hooley} C. Hooley, On the location of the roots of polynomial congruences, Glasgow Math. J. 32
(1990), 309--316.

\bibitem{jam} G. Jameson, Euler-Maclaurin, harmonic sums and Stirling's formula. The Mathematical Gazette, 99 (2015), 75-89.

\bibitem{mont} H. Montgomery, Ten lectures on the interface between analytic number theory and harmonic analysis (No. 84), American Mathematical Society, 1994.

\bibitem{nagell} T. Nagell, Eine Eigenschaft gewisser Summen. Skr. Norske Vid. Akad. Kristiania 13 (1923), p. 10--15.
 
\bibitem{olds} C. Olds, The Simple Continued Fraction Expansion of e, The American Mathematical Monthly, Vol. 77, No. 9 (Nov., 1970), pp. 968-974 (7 pages)

\bibitem{osler} T. Osler, A Proof of the Continued Fraction Expansion of $e^{1/M}$, The American Mathematical Monthly, Vol. 113, No. 1 (Jan., 2006), pp. 62-66

\bibitem{three} A. Schinzel, Extensions of Three Theorems of Nagell, Bull. Pol. Acad. Sci. Math. 61 (2013), 195-200.

\bibitem{theis} L. Theisinger, Bemerkung \"uber die harmonische Reihe, Monatsh. Math. 26 (1915), 132-134.

\bibitem{two} C. Wang and S. Hong, On the integrality of the elementary symmetric functions of $1, 1/3,..., 1/(2n-1)$. Mathematica Slovaca, 65(5), p. 957-962.

\end{thebibliography}
\end{document}